\documentclass[reqno,a4paper,12pt]{amsart}
\usepackage[hidelinks]{hyperref} 
\usepackage{amsmath, amssymb, amsfonts, amsbsy, amsthm, latexsym, epsfig}
\usepackage[backend=bibtex, style=numeric, url=false, isbn=false, maxbibnames=7, giveninits=true]{biblatex}
\usepackage[marginratio=1:1, margin = 2.3cm]{geometry}
\usepackage[shortlabels]{enumitem}
\usepackage{microtype}
\usepackage{mathtools}
\usepackage{graphicx,enumitem}
\mathtoolsset{showonlyrefs}
\usepackage{color}
\theoremstyle{plain}
\newtheorem{theo}{Theorem}[section]
\newtheorem{lemma}[theo]{Lemma}
\newtheorem{prop}[theo]{Proposition}

\newtheorem{rema}[theo]{Remark}

\numberwithin{equation}{section}

\def\Z{\mathbb{Z}}
\def\R{\mathbb{R}}

\def\supp{\text{Supp}}

\DeclareMathOperator*{\esssup}{ess\,sup}
\DeclareMathOperator*{\essinf}{ess\,inf}

\newcommand{\su}{\Delta_t}
\newcommand{\lu}{\Delta_\alpha}
\newcommand{\bu}{\Delta_\Omega}
\newcommand{\pu}{\Delta_{\mathrm{past}}}
\newcommand{\fu}{\Delta_{\mathrm{future}}}

\usepackage{lipsum}

\title[Model agnostic IF encoding]{Model agnostic signal encoding by leaky integrate-and-fire, performance and uncertainty}

	\author{Diana Carbajal}
	\address[D. Carbajal]{ Faculty of Mathematics, University of Vienna, Oskar-Morgenstern-Platz 1,
		A-1090 Vienna, Austria}
	\email{diana.agustina.carbajal@univie.ac.at}
	
	\author{Jos\'{e} Luis Romero}
	\address[J. L. Romero]{Faculty of Mathematics, University of Vienna,Oskar-Morgenstern-Platz 1, A-1090 Vienna, Austria, and Acoustics Research Institute, Austrian Academy of Sciences, Dr. Ignaz Seipel-Platz 2,	AT-1010 Vienna, Austria}
	\email{jose.luis.romero@univie.ac.at}

\thanks{D.C. was supported by the European Union’s programme Horizon Europe, HORIZON-MSCA-2021-PF-01, Grant agreement No. 101064206. 
This research was funded in
whole or in part by the Austrian Science Fund (FWF): 10.55776/Y1199. For
open access purposes, the authors have applied a CC BY public copyright
license to any author-accepted manuscript version arising from this
submission.}



\begin{document}
\begin{abstract}
Integrate-and-fire is a resource efficient time-encoding mechanism that summarizes into a signed spike train those time intervals where a signal's charge exceeds a certain threshold. We analyze the IF encoder in terms of a very general notion of approximate bandwidth, which is shared by most commonly-used signal models. This complements results on exact encoding that may be overly adapted to a particular signal model. We take into account, possibly for the first time, the effect of uncertainty in the exact location of the spikes (as may arise by decimation), uncertainty of integration leakage (as may arise in realistic manufacturing), and boundary effects inherent to finite periods of exposure to the measurement device. The analysis is done by means of a concrete \emph{bandwidth-based Ansatz} that can also be useful to initialize more sophisticated model specific reconstruction algorithms, and uses the earth mover's (Wasserstein) distance to measure spike discrepancy.
\end{abstract}

\maketitle

\section{Introduction}
\subsection{Time-encoding via integrate-and-fire}
Classical sampling schemes for measuring continuous-time signals are based on synchronous behavior: the signal's amplitude is measured at predetermined –uniform or irregular– instants, governed by a global clock. In many practical scenarios, conventional analog-to-digital converters (ADC) can lead to deficient implementations due to their size and power consumption. This is notably the case with brain computer interfaces based on electroencephalography (EEG) data, where special electrodes are located across the subject's scalp to sense the brain's neuronal activity: as traditional ADC devices are relatively bulky, they must be located outside of the head, thereby limiting the subject's mobility.

The critical need for compact and energy-efficient measuring devices has driven the research of asynchronous, event-driven sampling methods; that is, methods which are not ruled by a clock-pattern but rather capture the times where significant events in a signal stream occur. Time-Encoding Machines (TEM) fall in this category \cite{LT03, miskowicz2018event}. Instead of recording the signal's amplitude, the output of a TEM consists of a sequence of spikes that represents the times where a preset comparator is triggered. Time encoding is a flourishing research field \cite{10171393} that reaches into novel applications such as neuromorphic cameras \cite{Gal20}. In this article, we will consider the Integrate-and-Fire (IF) encoder. Roughly speaking, this TEM consists of an integrator followed by a comparator and simulates in a simplified manner the behavior of a neuron: action potentials (spikes) are generated when the accumulated stimulus (integrator) exceeds a certain threshold (comparator), while the accumulated charge wears out over time (leakage or memory loss); see Figure \ref{fig_if}. 

The seminal papers \cite{LT03,LT04,La05} introduced IF-TEM inspired by asynchronous delta-sigma modulators as an alternative sampling method to recover the amplitude of a bandlimited signal. Crucially, the sampling scheme in \cite{LT03,LT04,La05} incorporates a large bias to the signal, which ensures that the encoder's output is sufficiently dense, surpassing the Nyquist rate dictated by the signal's bandwidth \cite{MR2883827, unser2000sampling}. This in turn facilitates the application of the classical theory of non-uniform sampling, and frame-theoretic iterative reconstruction algorithms \cite{MR1875680, MR2516635}. The trade-off in incorporating a bias adder to the sampler is that the encoder will continue producing spikes even during inactive periods of the input signal. Although exactly encoding all bandlimited signals without applying a bias adder is impossible, as the signal's total energy could be too small with respect to the firing threshold to produce a single spike \cite{FPRAV12}, this is expected to be essentially the only obstruction to encoding. In fact, it
was shown in \cite{FPRAV12} that (leaky) IF-TEM without additive bias \emph{approximately encode} bandlimited functions, with uniform-norm errors that are comparable to the firing threshold.

More recently, it has been shown that a sequence of iterative projections onto adequately chosen convex sets (POCS method) can refine the reconstruction of bandlimited functions from IF samples by enforcing sampling consistency \cite{8361846, TRM22,TRM23}, that is, by searching for a signal that would produce the observed output. Time-encoding theory has also been extended to multi-channel TEM architectures \cite{La05b, ASV19, ASV20}, which have important applications in video time encoding \cite{ASV22} and spiking neural networks \cite{Ada22, ML24}. Moreover, applications of time encoding have been found in the emerging field of unlimited sampling \cite{FB22}.

Considerable effort has been devoted to obtaining comparable results for non-bandlimited signal models. In \cite{GV14} it is shown that perfect reconstruction from IF-TEM is possible for signals belonging to \emph{shift-invariant spaces} \cite{algr01} with suitably smooth generators, as long as the firing rate is high enough. To achieve the desired firing rate, the authors of \cite{GV14} propose replacing the (constant) threshold of the comparator with a sufficiently oscillating test function, a task that may be challenging for common architectures. An alternative algorithm to the one in \cite{GV14} was later proposed in \cite{FC15}, and further extended to the so-called neuromorphic sampling in \cite{KS23b}. Similarly, the approximate reconstruction result of \cite{FPRAV12} has been extended to shift-invariant spaces in \cite{JWC17}. 

Integrate-and-fire time-encoding machines have also been investigated in the important context of \emph{signals with finite rate of innovation} (FRI) \cite{MR1930786} --- which are linear combinations of shifts of certain kernels, with free and partially unknown translation nodes. In \cite{AD20,ATRD20,HAD20,HAD21,NME22,HD23}, the authors considered IF-TEM with no bias adder and investigated \emph{exact encoding} for several FRI signal models, including streams of Diracs, piecewise constant functions, and other spline-like kernels. Practical recovery methods that are independent of the FRI kernel have been introduced in \cite{RKS20,Flo23b,FB23}.

\begin{figure}
    \includegraphics[scale = 0.3]{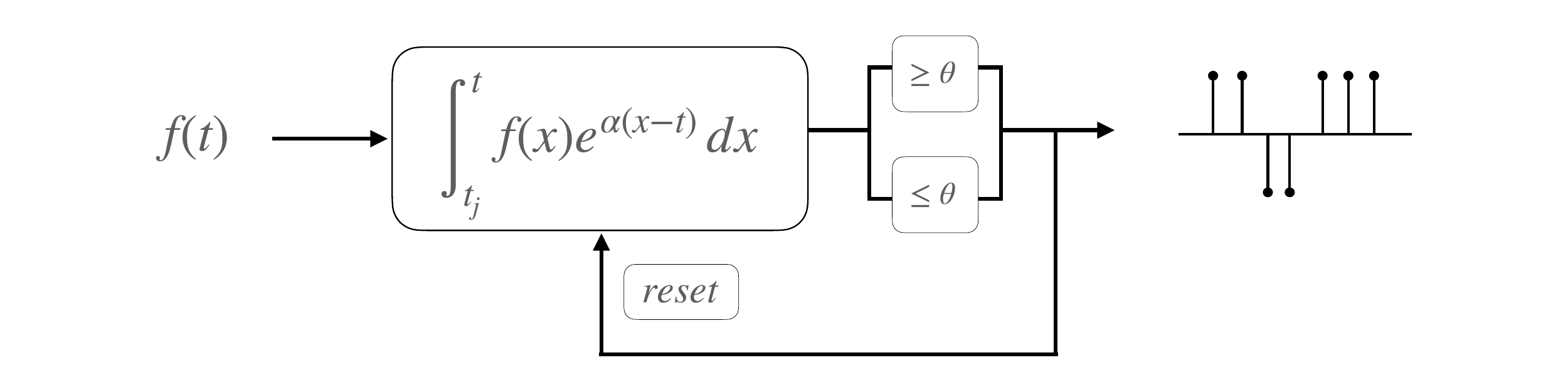}
    \caption{Block diagram of a leaky IF sampler with firing threshold $\theta>0$ and leakage intensity $\alpha>0$.}\label{fig_if}
\end{figure}

\subsection{Contribution}
In this article, we revisit the fundamental question of the performance of the leaky IF as a signal encoder, that is, we investigate whether different signals can be effectively distinguished from their IF outputs. Results to the effect that all signals within a certain class are perfectly encoded by IF provide only a preliminary answer. In fact, such uniqueness claims are \emph{fragile} in the sense that they cease to be (exactly) true as soon as signals are allowed to lie slightly outside the model studied. Intuitively, this is clear as IF provides just a summary of the signal's shape with an accuracy determined by the quality of the sampling device and particularly by the firing threshold.

While the basic intuition behind IF encoding is very general, many arguments used to validate it rely on delicate methods (such as Prony-like algorithms or iterative equation solvers) that may be particular to one specific signal model and sometimes are overly adapted to it. 
(In the literature, sometimes the signal model is even adapted to the IF output, by postulating a number of degrees of freedom that is equal to the number of output spikes.) Such approaches are very valuable because they do more than proving encoding: they provide practical and useful reconstruction algorithms. However, the question then arises whether the derived encoding guarantees are stable under slight model mismatches.

In this article we analyze the IF sampler by means of a very general notion of \emph{approximate bandwidth} which is shared by all commonly-used signal models. In this way, we provide \emph{model-agnostic approximate encoding guarantees} that
show that sufficiently different signals cannot have the same or very similar IF outputs. We do this by exploring the approximation properties of what we call the \emph{bandwidth-based Ansatz}, a concrete approximate decoder that can also be used to initialize more sophisticated model-specific reconstruction algorithms, such as POCS \cite{8361846, TRM22, TRM23}.

Most aspects of our analysis are novel even for bandlimited signals. We take the point of view that signals are analog while our access to them is finite. In the literature, it is common to adapt the signal model to the time-interval where the observations are made, or to assume that the unseen portion of the signal is small. In contrast, we explicitly incorporate uncertainty about the signal's past and future into the encoding guarantees in the form of a gap between the \emph{observation window} and the \emph{inference window}, where reliable prediction is possible. Second, we account for uncertainties in the amount of accumulation leakage (which is sometimes known only up to manufacturing specifications) and in the exact position of the firing spikes (which may be impacted by decimation and grid effects). To the best of our knowledge this is the first work to consider such effects for the IF sampler.

From a more general point of view, our analysis is in line with the investigations of \cite{Mo17,ML19,ML24} that introduce subtle metrics so that the IF map becomes almost isometric. In contrast, we focus on inverse stability (from output to input) with respect to easily interpretable metrics, such as the earth mover's distance (Wasserstein-1 metric). Our work is also technically related to the so-called \emph{cumulative transform} \cite{MR4301193, MR3842648}, which is a signal encoding method based on optimal transport.

\section{Results}
\subsection{The leaky integrate-and-fire sampler}
As signal, we consider a bounded measurable function
$f: \mathbb{R} \to \mathbb{C}$. Although the signal may potentially have infinite support, our access to it is finite, as the sampler is active only during an interval of length $T>0$. For simplicity we assume that the sampling device starts functioning at time $t=0$, and thus acquisition occurs within the \emph{observation window} $[0,T]$. The sampler is determined by two additional parameters: the \emph{firing threshold} $\theta>0$ and the \emph{leakage intensity} $\alpha>0$ that models the amount of energy that the accumulator loses over time.

The \emph{output} of the IF sampler consists of a (possibly empty) set of \emph{firing times} $\{t_1,\ldots, t_n\}$ with
\begin{align}\label{eq_ft}
0 \leq t_1 < \ldots < t_n \leq T
\end{align}
and corresponding \emph{firing phases} $\{q_1,\ldots, q_n\} \subset \mathbb{C}$ with $|q_j|= 1$; see Figure \ref{fig_exp1}.

The sampler is specified recursively: set $t_0 := 0$ and suppose that $t_j$ has been defined for $j<k$; we let $t_k$ be the minimum number $\geq t_{k-1}$ such that
\begin{align}\label{eq_firing}
\left|\int_{t_{k-1}}^{t_{k}} f(x)e^{\alpha(x-t_k)}\,dx\right| = \theta,
\end{align}
provided that such a number exists. Otherwise, the process stops. If the sampler does not fire at all, we let $n=0$. It is easy to see that the set of firing instants is always finite; indeed,
\begin{align}\label{eq_bound_n}
n \leq 1+\frac{1}{\theta} \int_0^T |f(x)|\,dx.
\end{align}
We also record the firing phases
\begin{align}\label{eq_phases}
q_j := \int_{t_{j-1}}^{t_{j}} f(x)e^{\alpha(x-t_j)}\,dx/\Big|\int_{t_{j-1}}^{t_{j}} f(x)e^{\alpha(x-t_j)}\,dx\Big|, \qquad j=1,\dots,n.
\end{align}
In most applications, $f$ is real-valued, and $q_j \in \{-1,1\}$. In this case, the output can be encoded as a signed \emph{train of spikes} \[\sum_{j=1}^n q_j \delta_{t_j},\] where $\delta$ is the Dirac distribution. For technical reasons it is convenient to allow for complex-valued signals and complex firing phases, but we shall use the jargon and intuition of the real-valued case.

\begin{figure}
    \includegraphics[scale = 0.3]{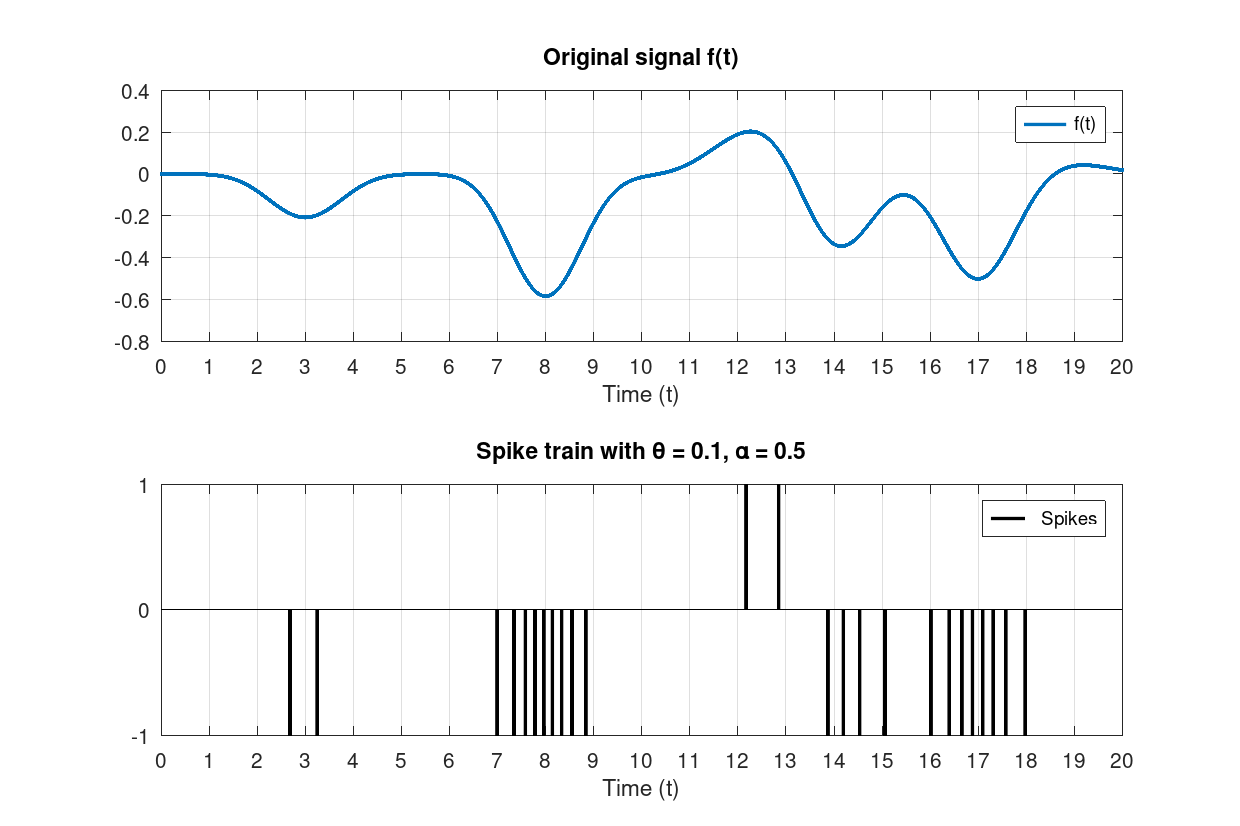}
    \caption{A signal and its IF output; see also Section \ref{sec_num}.}\label{fig_exp1}
\end{figure}

\subsection{Goal}
We aim to estimate a signal $f$ from its IF output. Although in principle it is reasonable to expect to be able to describe $f$ approximately in the observation window $[0,T]$, due to certain uncertainty factors, in general we will only draw conclusions on a smaller \emph{inference window} $[T_1,T_2] \subset [0,T]$.

\subsection{Sources of uncertainty}\label{sec_u}
\subsubsection{Specifications for the leakage intensity}
In realistic setups, the sampler's leakage intensity $\alpha$ conforms to a certain manufacturing specification
\begin{align}\label{eq_alpha_spec}
\alpha \in [\alpha_0-\lu, \alpha_0+\lu],
\end{align}
but may not be exactly known. Rather, only the \emph{leakage specification interval} $[\alpha_0-\lu, \alpha_0+\lu]$ with $\alpha_0-\lu>0$ is known. In case of perfect information, we let the \emph{leakage uncertainty} $\lu$ be $0$.

\subsubsection{The signal's past and future}
As the IF sampler encodes integral differences, its output provides relative information, which can only be turned into unconditional values when complemented with the signal's (unknown) past values $f(t)$, $t<0$. On the other hand, the accumulation leakage indicates that the influence of the signal's past should decay exponentially in time. To model this source of uncertainty, we introduce the \emph{exponentially dampened signal's past charge}
\begin{align}\label{eq_pu}
\pu = \pu(f) := \sup_{t \leq 0, |\alpha'-\alpha_0|\leq \lu} \Big| \int_{-\infty}^t f(x) e^{\alpha'(x-t)}\,dx \Big|.
\end{align}
In the literature, it is often assumed that $f(t)=0$ for $t<0$, in which case $\pu=0$,
or that $\pu$ is small with respect to the desired estimation accuracy. However, these assumptions are not completely justified in practice, and we prefer to explicitly model uncertainty about the past. 
Based on an estimate of the maximum signal size one can give the following safe bound for $\pu$:\footnote{We let $\|f\|_{L^\infty(E)} = \inf\{\lambda>0: |f(x)| \leq \lambda, \mbox{ a. e. }x\in E\}$ 
and write $\|f\|_\infty$ for short, when the underlying set $E$ is the full domain of $f$.}
\begin{align}\label{eq_pu_safe}
\pu \leq \frac{\|f\|_\infty}{\alpha_0 - \lu}.
\end{align}
Similarly, we define the \emph{exponentially dampened signal's future charge} as
	\begin{align}\label{eq_fu}
		\fu = \fu(f) := \sup_{t>t_n,  |\alpha'-\alpha_0|\leq \lu}\left|\int_{t_n}^t f(x)e^{\alpha' (x-t)}\,dx\right|,
	\end{align}
 which also satisfies the conservative bound:
	\begin{align}\label{eq_fu_safe}
		\fu \leq \frac{\|f\|_\infty}{\alpha_0 - \lu}.
	\end{align}
Moreover, if the sampler does not fire anymore after the firing time $t_n$, as often assumed in the literature (complete firing process), then $\fu \leq \theta$.

As we shall see, the uncertainties about the signal's past and future affect the inference performance only moderately, the key quantity being the \emph{logarithmic inference uncertainty} at level $\theta$:
\begin{align}\label{eq_sigma}
\sigma = \sigma(f,\theta) = \log^+\left(\frac{\max\{\pu,\fu\}}{\theta}\right).
\end{align}
(Here and henceforth we write $\log^+(x)=\max\{0,\log(x)\}$.)
 
\subsubsection{The location of the spikes}
While we will assume exact knowledge of the total number of firing times $n$ and the phases $q_1,\ldots,q_n$ --- which in practice are $\pm 1$ and thus robust to errors --- we shall allow for a certain amount of uncertainty in the exact location of spikes so as to account for sources of error such as decimation and latency. Precisely, we shall assume access to
\emph{approximate firing times} $0\leq t'_1 < \ldots < t'_n \leq T$ with total \emph{spike uncertainty}
\begin{align}\label{eq_su}
\su := \sum_{j=1}^n |t_j - t'_j|.
\end{align}
Up to normalization, the spike uncertainty $\su$ is the earth mover's (Hutchinson, Wasserstein-1) distance between the unsigned, normalized spike trains $\frac{1}{n}\sum_{j=1}^n \delta_{t_j}$ and $\frac{1}{n}\sum_{j=1}^n \delta_{t'_j}$, which measures the cost of transporting one point configuration into the other \cite{MR2459454}; see \eqref{eq_wa}.

\subsection{Signal's bandwidth}
We aim for estimation guarantees that are \emph{model agnostic} and only based on a very general notion of bandwidth. We say that a signal $f: \mathbb{R} \to \mathbb{C}$ has \emph{bandwidth $\Omega$ up to tolerance $\bu$} if
\begin{align}
\int_{|\xi| > \Omega} |\hat{f}(\xi)| \,d\xi \leq \bu.
\end{align}
Here $\Omega, \bu>0$ and the Fourier transform is normalized by
$\hat{f}(\xi)=\int_{\mathbb{R}} f(x) e^{-2\pi i \xi x}\,dx$. We also say that $f$ is approximately bandlimited to $[-\Omega,\Omega]$, and will be mostly interested in a tolerance level equal to the IF sampling threshold $\bu=\theta$.

The usual bandlimited functions, $\hat{f}(\xi)=0$, $|\xi|> \Omega$, have of course approximate bandwidth $\Omega$ with tolerance $\bu=0$. A second example is given by functions in a \emph{shift-invariant space}
\begin{align}\label{eq_intro_f}
f(x) = \sum_{j=1}^N \sum_{k \in \mathbb{Z}} c_{j,k} \varphi_j(x-k),
\end{align}
where $\varphi_1, \dots, \varphi_N$ are suitable generator functions, such as splines. As we show in Section \ref{sec_models}, such functions are approximately bandlimited, while the quality of the generators $\varphi_j$ determines the relation between the signal's energy $\|f\|_2$ and its approximate bandwidth $\Omega$. We also estimate the bandwidth of variants of \eqref{eq_intro_f} with free translation nodes (signals with finite rate of innovation); see Propositions \ref{propo:example-SIS} and \ref{prop_2}.

\subsection{The bandwidth-based Ansatz}
	Suppose that the IF output of $f$ is known up to the sources of uncertainty described in Section \ref{sec_u} and let $\Omega>0$ be a number called \emph{target bandwidth}.

\begin{figure}
\includegraphics[scale=0.7]{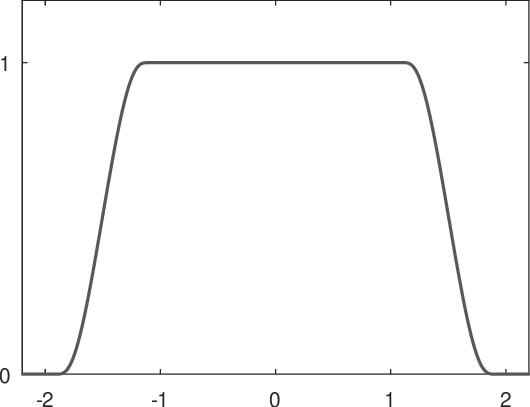}
\caption{The frequency cut-off function $\hat\psi$}
\label{fig_psi}
\end{figure}

	Let us select a \emph{frequency cut-off}: let $\psi: \mathbb{R} \to \mathbb{C}$ be a smooth function such that
	\begin{align}\label{eq_A1}
    \begin{cases}
		 &|\psi(x)|, |\psi'(x)| \leq C e^{-\sqrt{|x|}}, \mbox{ for } x \in \mathbb{R},\\
   		&\hat\psi(\xi) = 0,  \mbox{ for } |\xi|>2,\\
		&\hat\psi(\xi) = 1,  \mbox{ for }|\xi| \leq 1,\\
        &|\hat\psi(\xi) - 1| \leq 1,  \mbox{ for }\xi \in \mathbb{R},
    \end{cases}
	\end{align}
	where $C>0$ is a constant; see \cite{DH98} or \hyperref[note1]{End Note 1} for constructions of such functions and Figure \ref{fig_psi} for a plot.     
	Set $\psi_\Omega(x) := \Omega \cdot \psi(\Omega x)$ and define the \emph{bandwidth-based Ansatz} as
	\begin{align}\label{eq_ba}
		f_\Omega := \theta
		\sum_{j=1}^{n}
		\sum_{k=1}^j q_k
		(\alpha_0\psi_\Omega +  \psi_\Omega')
		* \big( e^{\alpha_0(t'_k-\bullet)} 1_{[t'_j,t'_{j+1})}\big),
	\end{align}
	where $\bullet$ indicates the function's argument, $*$ is the convolution product, and $t'_{n+1}=\infty$.

Our first result quantifies the performance of the bandwidth-based Ansatz.
    
    \begin{theo}\label{thm:main}
		Let a signal $f$ with $\hat{f} \in L^1(\mathbb{R})$ be measured on the observation window $[0,T]$ by an integrate-and-fire sampler with firing threshold $\theta$.
		
		\begin{itemize}
			\item Suppose that the firing times \eqref{eq_ft} are known with uncertainty $\su$, cf. \eqref{eq_su}, while the firing phases are known exactly.
			
			\item Suppose that the leakage parameter $\alpha$ satisfies the specification $\alpha \in [\alpha_0-\lu, \alpha_0+\lu]$, with $\alpha_0-\lu>0$.
			
			\item Let $\sigma$ be the logarithmic inference uncertainty \eqref{eq_sigma}.
		\end{itemize}
		Let $f$ have bandwidth $\Omega \geq 1$ with tolerance $\bu=\theta$ and define the inference window $[T_1,T_2]$ by
  \begin{equation}\label{eq_iw}
			T_1 = {\tfrac{\sigma}{\alpha_0-\Delta_\alpha}}+\tfrac{4}{\Omega}\sigma^2 \quad \text{and}\quad 
			T_2 = T - \tfrac{4}{\Omega} \sigma^2.
		\end{equation}
 Then the bandwidth-based Ansatz \eqref{eq_ba} satisfies
\begin{align}\label{eq_xc_0}
 \|f-f_\Omega\|_{L^\infty[T_1,T_2]} &\leq
 C \theta  \left(\alpha_0 + \lu +\Omega\right) \big[1+
			 n\tfrac{\lu}{\alpha_0 - \lu}
			+\su \cdot \left(\alpha_0 + n\Omega\right)\big],
\end{align}
for an absolute constant $C>0$.
 \end{theo}
Theorem \ref{thm:main} is proved in Section \ref{sec_proof}; we now make some remarks.

\begin{itemize}[left=0pt, itemsep=1em]
\item The assumption that $\Omega \geq 1$ is only made to obtain a simpler estimate; see Theorem \ref{thm:main_plus} below for a more general statement. 

\item In most situations the signal's bandwidth is large in comparison to the (maximal) leakage parameter, say 
\begin{align}\label{eq_under}
\alpha_0 + \lu \leq \Omega,
\end{align}
which leads to the simplified error bound
\begin{align}\label{eq_xc_simple}
 \|f-f_\Omega\|_{L^\infty[T_1,T_2]} &\leq
 C \theta \Omega \big[1+ n\tfrac{\lu}{\alpha_0 - \lu}+\su n\Omega\big].
\end{align}
This supports the basic intuition that an approximate reconstruction is possible if $\theta \Omega \ll 1$.

\item (\emph{Accuracy-coverage trade-off})\footnote{Here, $O(\theta)$ denotes a quantity bounded by $C\theta$, where $C$ is a suitable constant.}
: In the regime $\theta \to 0$, the gap between the observation window $[0,T]$ and the inference window $[T_1,T_2]$ is roughly $\log^2(1/\theta)=\log^2(\theta)$, which is moderate in comparison to the corresponding estimation accuracy $O(\theta)$.

\item (\emph{Leakage and spike uncertainties}): To the best of our knowledge, Theorem \ref{thm:main} is the first formal result that accounts for uncertainties in the leakage parameter $\alpha$ (which may be known only up to manufacturing specifications) and the firing times $t_1, \ldots, t_n$ (which may be impacted by decimation and grid effects). Theorem \ref{thm:main} shows that a level of uncertainty\footnote{The meaning of \eqref{eq_un} is that $\lu, \su \leq C_* \theta$ for a certain constant $C_*$, which impacts the constant in \eqref{eq_un_xxx}.}
\begin{align}\label{eq_un}
\lu, \su \lesssim \theta
\end{align}
leads to estimates roughly similar to the ones for the ideal case $\lu=\su=0$. Indeed, under \eqref{eq_under} and \eqref{eq_un}, the conservative bound for the number of spikes \eqref{eq_bound_n} helps reduce \eqref{eq_xc_simple} to
\begin{align}\label{eq_un_xxx}
 \|f-f_\Omega\|_{L^\infty[T_1,T_2]} &\leq
 C \theta \Omega,
\end{align}
for a constant that depends on the minimal leakage $\alpha_0-\lu$ and the signal's charge $\int_0^T |f|$.

\item (\emph{Infinite observation window}): If the observation window is infinite, one can let $T \to \infty$ in Theorem \ref{thm:main} and obtain a conclusion valid on $[T_1,\infty)$. If, moreover, $f$ is assumed to be integrable on $[0,\infty)$, as often done in the literature, then \eqref{eq_bound_n} shows that the IF output is finite. As a consequence $f_\Omega$ is independent of $T$ for $T$ large, and $\fu \leq \theta$. 

If, in addition, one assumes that $\pu \leq \theta$ (as often done in the literature), then $T_1=0$ and we obtain a bound valid in $[0,+\infty)$, which, under \eqref{eq_under}, reads
\begin{align}
\|f-f_\Omega\|_{L^\infty([0,\infty))} \leq  C \theta \Omega \big[1+ n\tfrac{\lu}{\alpha_0 - \lu}+\su n\Omega\big].
\end{align}
The formulation in Theorem \ref{thm:main} is more flexible: any bounds on $\pu$ and $\fu$--- such as the conservative \eqref{eq_pu_safe}, \eqref{eq_fu_safe} --- can be used, and their quality (moderately) impacts our ability to learn the signal.

\item (\emph{Initialization of iterative algorithms}): As with the reconstruction formula for bandlimited functions \cite{FPRAV12}, the bandwidth-based Ansatz \eqref{eq_ba} provides a starting point for iterative algorithms that exploit model specific properties, such as those based on projections onto adequate convex sets (POCS) \cite{8361846, TRM22, TRM23}.

\item To instantiate the bound in Theorem \ref{thm:main}, one needs estimates on the bandwidth $\Omega$ that achieves the tolerance $\bu=\theta$. In Section \ref{sec_models} we provide such estimates for concrete signal models.
\end{itemize}

Section \ref{sec_num} reports on numerical experiments that illustrate Theorem \ref{thm:main}. The results are consistent with \eqref{eq_xc_0} with respect to the threshold level $\theta$ and the spike uncertainty $\lu$, while robustness to leakage uncertainty appears to be even stronger than what \eqref{eq_xc_0} implies.

\subsection{Signal models and quantification of their uncertainty}
In most of the literature on the IF sampler, a specific signal model is assumed and the proposed reconstruction method depends heavily on that specific assumption. For example, if the signal is assumed to be a linear combination of shifted spline-like kernel functions (with free translation nodes), then Prony-like methods can be leveraged to show that signals are \emph{exactly} encoded by the IF sampler. However, this analysis is only effective if the number of translation nodes does not exceed the number of fired spikes (which may be unsatisfactory from the point of view of modeling) and the numerical impact of the spatial distribution of the translation nodes may be unclear from such arguments.

The bandwidth-based Ansatz \eqref{eq_ba} on the other hand only depends on the very general notion of approximate bandwidth, which is shared by many models (see Section \ref{sec_models}). To be precise, consider a IF threshold level $\theta$ and a (non-linear) class of functions $\mathcal{C}$ such that:
\begin{itemize}
\item[(M1)] Each $f \in \mathcal{C}$ has an integrable Fourier transform, and bandwidth $\Omega=\Omega_{\mathcal{C}} \geq 1$ up to tolerance $\theta$.
\item[(M2)] The past \eqref{eq_pu} and future  \eqref{eq_fu} uncertainties of each $f \in \mathcal{C}$ are bounded by a common number $\Delta_{\mathcal{C}}$.
\end{itemize}
Further, set $\sigma=\sigma_{\mathcal{C}} := \log^+(\Delta_{\mathcal{C}}/\theta)$. As mentioned before, assumptions on the past and future uncertainty such as (M2) are often implicitly made either by imposing suitable support conditions, or by assuming that the signal is (finitely) integrable and that the IF sampling process is complete (no more firing is possible). Without such assumptions the safe bounds \eqref{eq_pu_safe} and \eqref{eq_fu_safe} can be used. The key assumption is the uniform bandwidth condition (M1), which is satisfied, for example, by spline-like signals with bounded quadratic mean (see Section \ref{sec_models}).

\begin{theo}\label{coro_1}
Let $\mathcal{C}$ be a class that meets the
bandwidth and uncertainty assumptions (M1), (M2). Let $f_1, f_2 \in \mathcal{C}$ be measured on the observation window $[0,T]$ by an integrate-and-fire sampler with firing threshold $\theta$, whose leakage parameter is known up to the specification $\alpha \in [\alpha_0-\lu, \alpha_0+\lu]$, $\alpha_0-\lu>0$, and satisfies \eqref{eq_under}.

Suppose that $f_1$ and $f_2$ produce $n$ spikes located at $\{t_1,\ldots,t_n\}$ and $\{t'_1,\ldots,t'_n\}$ respectively and the same firing phases $q_j$. Set $\su := \sum_{j=1}^n |t_j - t'_j|$ and define the inference window $[T_1,T_2]$ by \eqref{eq_iw}. Then
\begin{align}\label{eq_fii}
 \|f_1-f_2\|_{L^\infty[T_1,T_2]} &\leq
 C \theta \Omega \big[1+ n\tfrac{\lu}{\alpha_0 - \lu}+\su n\Omega\big].
\end{align}
\end{theo}
Theorem \ref{coro_1} is proved in Section \ref{sec_coro_1} by comparing each signal to its bandwidth Ansatz. We make some remarks.
\begin{itemize}[left=0pt, itemsep=1em]
\item Theorem \ref{coro_1} provides an \emph{approximate encoding guarantee} by showing that significantly different signals cannot have the same or very similar IF outputs. 

\item If we encode the spike trains produced by $f_1,f_2$ into probability measures
\begin{align}
\mu_1=\frac{1}{n}\sum_{j=1}^n \delta_{t_j}, \qquad
\mu_2=\frac{1}{n}\sum_{j=1}^n \delta_{t'_j},
\end{align}
then, after normalization, the spike uncertainty is the earth mover's (Wasserstein-1, Hutchinson) distance:
\begin{align}\label{eq_wa}
\tfrac{1}{n} \Delta_t = W_1(\mu_1,\mu_2) = \inf_{\pi} \sum_{j,k=1}^n \pi_{j,k} |t_j-t'_k|,
\end{align}
where the infimum is taken over all probability measures $\pi$ on the plane with marginals $\mu_1, \mu_2$, and $\pi_{j,k}=\pi(\{(t_j,t_k)\})$; see, e. g., \cite[Lemma 4.2]{MR4028181}. The stability guarantee \eqref{eq_fii} is substantially stronger than the analogous estimate involving the total variation distance, which is not continuous under small perturbations of the points $t_j$.

\item The encoding guarantees of Theorem \ref{coro_1} complement the more abstract investigations of
\cite{Mo17, ML19, ML24} where input and output metrics are defined so that the IF map becomes almost isometric. In contrast, we focus on inverse stability (from output to input) with respect to concrete metrics such as \eqref{eq_wa}.
\end{itemize}

The remainder of the article is organized as follows: Section \ref{sec_proof} presents a proof of Theorems \ref{thm:main} and Theorem \ref{coro_1} (together with the more technical
Theorem \ref{thm:main_plus}); Section \ref{sec_models} provides examples of signal models with uniform approximate bandwidth, which help illustrate Theorem \ref{coro_1}; Section \ref{sec_num} presents numerical experiments; Section \ref{sec_con} provides final conclusions. The end notes in Section \ref{sec_notes} contain additional technical comments. 

\section{Proof of main result}\label{sec_proof}
Recall that we write $\log^+(x):=\max\{\log(x),0\}$ and $\log^+(0)=0$. The Fourier transform of a function $f$ is denoted $\mathcal{F}f(\xi) = \hat{f}(\xi) = \int_{\mathbb{R}} f(x) e^{-2\pi i x \xi} dx$.

For a given set $\tau=\{t_1,\dots,t_n\}$ and a parameter $\alpha>0$, let us introduce the \emph{potential function}:
	\begin{equation}\label{eq:u}
		u_{\tau,\alpha}(t) := \theta\left(
		\sum_{j=1}^{n-1}
		\sum_{k=1}^j q_k  e^{\alpha(t_k-t)} 1_{[t_j,t_{j+1})}(t) + \sum_{k=1}^n q_k e^{\alpha(t_k-t)} 1_{[t_n,+\infty)}(t)\right), \qquad t \in \mathbb{R}.
	\end{equation}
Throughout the remainder of this section, we assume that $f$ is a function with $\hat{f} \in L^1(\mathbb{R})$ that has been measured on the observation window $[0,T]$ by an integrate-and-fire sampler with firing threshold $\theta$, producing firing times \eqref{eq_ft} and phases
\eqref{eq_phases}. Suppose also that $0 \leq t'_1 < \ldots < t'_n \leq T$ with \eqref{eq_su}, and $\alpha \in [\alpha_0-\lu, \alpha_0+\lu]$, with $\alpha_0-\lu>0$. Define $\sigma$ by \eqref{eq_sigma}. The quantities $T_1, T_2$ are defined by \eqref{eq_iw}, as soon as $\Omega$ is specified.

	\subsection{Sensitivity to uncertainties}
	
	\begin{lemma}[Sensitivity to the leakage intensity]\label{lem:leakage}
		Let $\tau = \{t_1,\dots,t_n\}$. Then
		$$\|u_{\tau,\alpha}-u_{\tau,\alpha_0}\|_{\infty}\leq \theta \cdot \lu \cdot  \frac{n}{\alpha_0 - \lu}.$$
	\end{lemma}
	
	\begin{proof}
    If $t <t_1$, then $u_{\tau,\alpha}(t) = u_{\tau,\alpha_0}(t)=0$. Let us now assume that $t\in [t_j, t_{j+1})$, with $j=1,\dots,n-1$.

    We use the following bound. For any $s \geq 0$, we use the mean-value theorem to select $\gamma\in(\alpha_0-\lu,\alpha_0+\lu)$ and estimate
    \begin{align*}\label{eq:mean-val}
			\left|e^{-\alpha s}-e^{-\alpha_0 s}\right| &= e^{-\gamma s} s|\alpha-\alpha_0| = \frac{1}{\gamma} |\alpha-\alpha_0| (\gamma s) e^{-\gamma s}
   \\
   &\leq
   \frac{1}{\gamma} |\alpha-\alpha_0| \leq
     \frac{\lu}{\alpha_0 - \lu},
		\end{align*} 
where we used that $xe^{-x} \leq 1$ for $x\geq 0$. Hence, 
		\begin{align}
			|u_{\tau,\alpha}(t) - u_{\tau,\alpha_0}(t)| 
			&=\theta \cdot \left|  \sum_{k=1}^j q_k  e^{\alpha(t_k-t)}  - \sum_{k=1}^j q_k  e^{\alpha_0(t_k-t)} \right|\\
			&\leq \theta \cdot \sum_{k=1}^j \left| e^{-\alpha(t-t_k)} - e^{-\alpha_0(t-t_k)}\right| \\
			&\leq \theta \cdot \lu \cdot  \frac{n}{\alpha_0 - \lu}.
		\end{align}
		The argument for $t\in [t_n,+\infty)$ is similar.
	\end{proof}

At this point, it is convenient to introduce the following \emph{Wiener amalgam norm} \cite{MR1544780} of a measurable function $g: \mathbb{R} \to \mathbb{C}$,
\begin{align*}
\|g\|_{W(L^1,L^\infty)} = \esssup_{x\in\mathbb{R}} \int_{x-1/2}^{x+1/2} |g(t)|\,dt,
\end{align*}
which controls the average size of $g$ per unit time.
 
	\begin{lemma}[Sensitivity to firing times]\label{lem:spikes} Let $\tau =\{t_1,\dots,t_n\}$, and $\tau'=\{t_1',\dots,t_n'\}$. Then
		$$\|u_{\tau,\alpha} - u_{\tau',\alpha}\|_{W(L^1,L^\infty)} \leq \su \cdot \theta \cdot \left(\alpha + 4 n\right).$$
	\end{lemma}	
	\begin{proof}
		Define, for $j=0,\dots,n-1$, the intervals $I_j = [t_j,t_{j+1})$ and $I'_j = [t'_j, t'_{j+1})$ ($t_0 = t'_0 = 0$), and $I_n = [t_{n},+\infty)$, $I_n' = [t'_{n}, +\infty)$. Then
		\begin{align}\notag
			\|u_{\tau,\alpha} - u_{\tau',\alpha}\|_{W(L^1,L^\infty)}:&=\esssup_{x\in \R} \int_{I^x} |u_{\tau,\alpha} (t) - u_{\tau',\alpha}(t)|\,dt \\\label{eq_a}
			&= \esssup_{x\in \R} \int_{I^x \cap [0,+\infty)} |u_{\tau,\alpha} (t) - u_{\tau',\alpha}(t)|\,dt,
		\end{align}
		where $I^x := x+[-1/2,1/2]$,  $x\in\R$. Define $S_j :=I_j \cap I'_j$ and $E_j = I_j \setminus I'_j$. We think of $S_j$ as the bulk of the interval $I_j$ and of $E_j$ as a remaining exceptional set. We have the decomposition $$[0,+\infty) =\bigcup_{j=0}^n S_j  \cup \bigcup_{j=0}^n E_j.$$
  To estimate the essential supremum in \eqref{eq_a} we fix $x\in \R$, and first bound the integrand in \eqref{eq_a} on the bulk. For $t\in S_j \cap I^x$ with $j=1,\dots,n$, by the mean-value theorem,
		\begin{align}
			|u_{\tau,\alpha} (t) - u_{\tau',\alpha}(t)| 
			&=\theta \cdot \left|  \sum_{k=1}^{j}e^{\alpha(t_k-t)} q_k - \sum_{k=1}^j e^{\alpha (t'_k-t)} q_k\right|\\
			&\leq \theta \cdot \sum_{k=1}^j \left| e^{\alpha(t_k-t)} - e^{\alpha(t'_k-t)}\right| \\
			&\leq \theta \cdot \sum_{k=1}^j |t_k-t_k'| \cdot \alpha \cdot \sup_{s \in [0,1]} e^{\alpha [s(t_k-t) +(1-s)(t'_k-t)]}
            \\
			&\leq \theta \cdot \su \cdot \alpha.
		\end{align}
		On the other hand, for the exceptional points $t\in E_j\cap I^x$ we use the crude bound $$|u_{\tau,\alpha} (t) - u_{\tau',\alpha}(t)|\leq 2n\theta.$$ We now observe that the measure of the exceptional sets add up to
		$$\left|\bigcup_{j=0}^n E_j \cap I^x\right| \leq 2\su.$$
		Thus, taking into account that $|I^x|=1$,
		\begin{align}
			\int_{I^x \cap [0,+\infty)} |u_{\tau,\alpha} (t) - u_{\tau',\alpha}(t)|\,dt 
			&= \int_{\bigcup_{j=0}^nS_j \cap I^x } |u_{\tau,\alpha} (t) - u_{\tau',\alpha}(t)|\,dt \\
			&\quad+ \int_{\bigcup_{j=0}^n E_j \cap I^x} |u_{\tau,\alpha} (t) - u_{\tau',\alpha}(t)|\,dt
			\leq \su \cdot \theta \cdot \left(\alpha + 4 n \right),\label{eq:fix-x}
		\end{align}
		from which the statement follows.
	\end{proof}
	
	\subsection{Approximate bandwidth \texorpdfstring{$\Omega = 1$}{Omega = 1} and no temporal uncertainty}
	
	\begin{lemma}\label{lem:perfect}
		Let $f$ have bandwidth $\Omega = 1$ with tolerance $\Delta_1$. Then
		\begin{equation}\label{eq_e}
			\|f - (\alpha \psi + \psi')*u_{\tau,\alpha}\|_{L^\infty[T_1,T_2]}\leq \Delta_1 + 28 C  (\alpha+1) \cdot \theta,
		\end{equation}
  where $C$ is the constant in \eqref{eq_A1}.
	\end{lemma}
	
	\begin{proof}
		Let us define 
		$$F_\alpha(t):=\int_{-\infty}^t f(x)e^{\alpha(x-t)}\,dx,\quad t\in\R,$$
		and bound
		\begin{align}
			\|f - (\alpha \psi + \psi')*u_{\tau,\alpha}\|_{L^\infty[T_1,T_2]}
			&\leq \|f- (\alpha \psi + \psi')*F_\alpha\|_{L^\infty[T_1,T_2]}\label{eq:bound-1} \\
			&\qquad+ \|(\alpha \psi + \psi')*(F_\alpha-u_{\tau,\alpha})\|_{L^\infty[T_1,T_2]}.\label{eq:bound-2}
		\end{align}
		\subsubsection*{Step 1} We first bound the term on the right-hand side of \eqref{eq:bound-1}. Since $\hat{F}_\alpha(\xi) = \frac{1}{\alpha+2\pi i \xi}\, \hat{f}(\xi),$
		it holds that
		$$\mathcal{F} ((\alpha \psi + \psi')*F_\alpha) (\xi) = \hat{f}(\xi) \hat{\psi}(\xi).$$
		Thus,
  \begin{align}\label{eq_f}
		\begin{aligned}
			\left\| f - (\alpha \psi + \psi')*F_\alpha \right\|_{\infty} &\leq \int_{\R} |\hat{f}(\xi)\hat{\psi}(\xi) - \hat{f}(\xi)|\,d\xi
			= \int_{\R} |\hat{f}(\xi)| \big|\hat{\psi}(\xi) - 1\big|\,d\xi\\
			&\leq  \int_{|\xi|>1} |\hat{f}(\xi)| \,d\xi
			\leq \Delta_1.
		\end{aligned}
  \end{align}
\subsubsection*{Step 2} We show that
		\begin{equation}\label{eq:error-in-tj}
			\left|F_\alpha(t_j)-u_{\tau,\alpha}(t_j)\right| = e^{-\alpha t_j }\left|F_\alpha(0)\right|, \qquad
   j=1,\dots,n.
		\end{equation}
        In terms of the phases $q_j$, the firing equation
        \eqref{eq_firing} reads
        \begin{align}
        \int_{t_{j-1}}^{t_{j}} f(x)e^{\alpha(x-t_j)}\,dx = q_j \theta, \qquad j=1,\dots,n.
        \end{align}
        Thus, for $j=1$, 
		\begin{align}
			\left| F_\alpha(t_1) - u_{\tau,\alpha}(t_1) \right| &=\left| e^{-\alpha t_1 }\int_{-\infty}^0 f(x)e^{\alpha x}\,dx +  \int_{0}^{t_1} f(x)e^{\alpha(x-t_1)}\,dx - \theta
			q_1  \right|
			=  e^{-\alpha t_1 }\left|F_\alpha(0)\right|.
		\end{align}
		For  $j=2,\dots,n$,
		\begin{align}
			\MoveEqLeft \left|F_\alpha(t_j) - u_{\tau,\alpha}(t_j)\right| \\
			&= \left| e^{\alpha(t_{j-1}-t_j)}\int_{-\infty}^{t_{j-1}} f(x)e^{\alpha(x-t_{j-1})}\,dx + \int_{t_{j-1}}^{t_j} f(x)e^{\alpha(x-t_{j})}\,dx - \theta \sum_{k=1}^{j} q_k e^{\alpha(t_k-t_j)} \right| \\
			&= e^{\alpha(t_{j-1}-t_{j})} \left|F_\alpha(t_{j-1})- u_{\tau,\alpha}(t_{j-1}) \right|,	\label{eq:w_j-Af(t_j)}
		\end{align}
		from which \eqref{eq:error-in-tj} follows.

  \subsubsection*{Step 3} Recall the logarithmic inference uncertainty \eqref{eq_sigma} and denote \[\tilde{T}_1:=\frac{1}{\alpha_0-\Delta_\alpha} \sigma.\]
		We will show that
		\begin{equation}\label{eq:error-in-t}
			\left\| u_{\tau,\alpha} - F_\alpha\right\|_{L^\infty[\tilde T_1,T]}\leq2\theta\quad\text{and}\quad\left\| u_{\tau,\alpha} - F_\alpha\right\|_{L^\infty(\R\setminus [\tilde T_1,T])}\leq\pu + \max\{\theta; \fu\}.
		\end{equation}
        The exponentially damped signal past \eqref{eq_pu} provides the bound
        \begin{align}\label{eq_d}
        |F_\alpha(t)|\leq \pu, \qquad t \leq 0.
        \end{align}
        Since $u_{\tau,\alpha} \equiv 0$ on $(-\infty,0)$ we thus have
        \begin{align}\label{eqa1}
        \left\| u_{\tau,\alpha} - F_\alpha\right\|_{L^\infty(-\infty,0)}\leq\pu.
        \end{align}
        We now look into the positive half-line and decompose it as
        \begin{align}
        [0,+\infty) = \bigcup_{j=0}^{n+1} I_j,
        \end{align}
        where $I_j = [t_j,t_{j+1})$ for $j=0,\dots,n-1$ ($t_0 = 0$), $I_n = [t_{n},T)$,
        $I_{n+1}=[T,\infty)$. Let $t\in I_j$, $j=0,\dots,n$, and use \eqref{eq:error-in-tj} and \eqref{eq_d} to estimate
		\begin{align}\label{eq_b}
  \begin{aligned}
			\Big|F_\alpha(t)- u_{\tau,\alpha}(t)\Big|
			& = \Big|e^{\alpha(t_j-t)}\int_{-\infty}^{t_j} f(x)e^{\alpha(x-t_j)}\,dx + \int_{t_j}^{t} f(x)e^{\alpha(x-t)}\,dx - \theta\sum_{k=1}^j q_k  e^{\alpha(t_k-t)} \Big|\\
            &=\Big|e^{\alpha(t_j-t)}\big[F_\alpha(t_j) - u_{\tau,\alpha}(t_j)\big] + \int_{t_j}^{t} f(x)e^{\alpha(x-t)}\,dx\Big|\\
			&\leq e^{\alpha(t_j-t)} \big|F_\alpha(t_j) - u_{\tau,\alpha}(t_j)\big| + \Big| \int_{t_j}^{t} f(x)e^{\alpha(x-t)}\,dx\Big| \\
			&\leq e^{-\alpha t}\big|F_\alpha(0)\big| + \Big| \int_{t_j}^{t} f(x)e^{\alpha(x-t)}\,dx\Big|\\
			&\leq e^{-\alpha t} \pu + \Big| \int_{t_j}^{t} f(x)e^{\alpha(x-t)}\,dx\Big|.
   \end{aligned}
		\end{align}
 For $t\in I_{n+1}$ we argue similarly to find that  $|F_\alpha(t)- u_{\tau,\alpha}(t)| \leq e^{-\alpha t} \pu + | \int_{t_n}^{t} f(x)e^{\alpha(x-t)}\,dx|$ and further use the exponentially damped signal future \eqref{eq_fu} to obtain
        \begin{align}\label{eqa2}
     \left\| u_{\tau,\alpha} - F_\alpha\right\|_{L^\infty[T,+\infty)}\leq \pu + \fu.
        \end{align}
Let us now consider a point in the observation window $t\in[0,T)$ and select $j \in \{0,\ldots,n\}$ with $t \in I_j$. The fact that the integrator did not fire at any time instant $s \in I_j^\circ$ implies that
  \begin{align*}
  \Big| \int_{t_j}^{t} f(x)e^{\alpha(x-t)}\,dx\Big| \leq \theta.
  \end{align*}
Thus, \eqref{eq_b} gives
\begin{align}\label{eq_c}
\Big|F_\alpha(t)- u_{\tau,\alpha}(t)\Big|
\leq e^{-\alpha t} \pu + \theta, \qquad 0 \leq t \leq T.
\end{align}
If, in addition, $t \geq \tilde{T}_1$, we have 
  \begin{align}
  \alpha t \geq (\alpha_0 - \lu) \tilde{T}_1 \geq   \log\big(\tfrac{\Delta_\text{past}}{\theta}\big),
  \end{align}
  so that $e^{-\alpha t} \pu \leq {\theta}$. Hence, 
  \begin{align}\label{eqa3}
 \left\| u_{\tau,\alpha} - F_\alpha\right\|_{L^\infty[\tilde T_1,T]}\leq2\theta.
 \end{align}
 On the other hand, \eqref{eq_c} implies
 \begin{align}\label{eqa4}
  \left\| u_{\tau,\alpha} - F_\alpha\right\|_{L^\infty[0,\tilde T_1)}\leq \pu + \theta.
 \end{align}
 We obtain \eqref{eq:error-in-t} by combining \eqref{eqa1}, \eqref{eqa2}, \eqref{eqa3} and \eqref{eqa4}.
 
		\subsubsection*{Step 4} Let us bound the term in \eqref{eq:bound-2}. We let $t \in \mathbb{R}$ and use \eqref{eq:error-in-t} to estimate
		\begin{align}
			\MoveEqLeft\left|(\alpha \psi + \psi')*(F_\alpha-u_{\tau,\alpha})(t)\right| \leq \int_{\R} |F_\alpha(x)-u_{\tau,\alpha}(x)|\cdot|(\alpha \psi + \psi')(t-x)|\,dx\\
			&\leq 2\theta \int_{\tilde T_1}^T |(\alpha \psi + \psi')(t-x)|\,dx 
			+ \int_{\R\setminus [\tilde T_1,T]} |F_\alpha(x)-u_{\tau,\alpha}(x)|\cdot|(\alpha \psi + \psi')(t-x)|\,dx
            \\
			&\leq 2\theta \|\alpha \psi + \psi'\|_1 + \big(\pu + \max\{\theta; \fu\} \big) \int_{\R\setminus [\tilde T_1,T]} |(\alpha \psi + \psi')(t-x)|\,dx
            \\
			&\leq 3\theta \|\alpha \psi + \psi'\|_1 + 2\max\{\pu, \fu\} \int_{\R\setminus [\tilde T_1,T]} |(\alpha \psi + \psi')(t-x)|\,dx.\label{eq:bound-3}
		\end{align}
        Using \eqref{eq_A1},
		\begin{equation}\label{eq:bound-4}
			\|\alpha\psi+\psi'\|_1 \leq C(\alpha+1) \int_{-\infty}^\infty e^{-\sqrt{|x|}}\,dx = 4C  (\alpha+1).
		\end{equation}
  Second, note that $T_1 = \tilde{T_1} + 4\sigma^2$ and $T_2 = T - 4\sigma^2$, so, if $t\in[T_1,T_2]$,
		\begin{align}
			\int_{\R \setminus [\tilde T_1,T]} |(\alpha\psi+\psi')(t-x)| \,dx 
			&\leq \int_{\R \setminus [-4\sigma^2,4\sigma^2]}  |(\alpha\psi+\psi')(y)| \,dy\\
			&\leq 2C(\alpha + 1) \int_{4\sigma^2}^{+\infty} e^{-\sqrt{y}}\,dy\\
			&= 4C(\alpha +1)\cdot\big(2\sigma+1\big)e^{- 2\sigma}\\
            &\leq 8C(\alpha +1)\cdot e^{-\sigma}.
		\end{align}
  Hence,
  \begin{align}
  \max\{\pu, \fu\} \int_{\R\setminus [\tilde T_1,T]} |(\alpha \psi + \psi')(t-x)|\,dx \leq 8C (\alpha +1)\cdot \theta.\label{eq:bound-5}
  \end{align} 
    Combining \eqref{eq:bound-3}, \eqref{eq:bound-4} and \eqref{eq:bound-5}, we conclude that
		\begin{align}
			\MoveEqLeft\left\|(\alpha \psi + \psi')*(F_\alpha-u_{\tau,\alpha})\right\|_{L^\infty[T_1,T_2]} \leq 28 C  (\alpha+1) \cdot \theta,
		\end{align}
        which provides a bound for \eqref{eq:bound-2}. On the other hand, \eqref{eq:bound-1} was estimated in \eqref{eq_f}, which proves \eqref{eq_e}.
	\end{proof}

 \subsection{Approximate bandwidth 1 with temporal uncertainty}
	\begin{prop}\label{thm:bw-1}
		Let $f$ have bandwidth $\Omega = 1$ with tolerance $\Delta_1$. Then the bandwidth-based Ansatz \eqref{eq_ba} satisfies
		\begin{align}\label{eq_xa}
			\MoveEqLeft\|f-f_1\|_{L^\infty[T_1,T_2]}\leq \Delta_1 + \theta \cdot 28C(\alpha_0 + \lu +1) \times \\ &\qquad\qquad\qquad\qquad\qquad\bigg[1
			+ \lu\cdot  \frac{n}{\alpha_0 - \lu}
			+\su \cdot K(\alpha_0 + n) \bigg]
		\end{align}
		where $C$ is the constant from \eqref{eq_A1} and $K>0$ is a universal constant. 
	\end{prop}
	
	\begin{proof}
		Let $f_1$ be the bandwidth-based Ansatz from \eqref{eq_ba}. Incorporating 
		the notation \eqref{eq:u}, we have that $f_1 = (\alpha_0\psi + \psi')* u_{\tau',\alpha_0}$.
		Our goal is to bound
		\begin{align}
			\|f-(\alpha_0\psi + \psi')* u_{\tau',\alpha_0}\|_{L^\infty[T_1,T_2]} &\leq \|f- (\alpha\psi + \psi')* u_{\tau,\alpha}\|_{L^\infty[T_1,T_2]}  \label{eq:teo-1}\\
			&\quad+\|(\alpha\psi + \psi')* u_{\tau,\alpha} - (\alpha_0\psi + \psi')* u_{\tau,\alpha_0}\|_{\infty}  \label{eq:teo-2}\\
			&\quad+\|(\alpha_0\psi + \psi')* (u_{\tau,\alpha_0} - u_{\tau',\alpha_0})\|_{\infty} \label{eq:teo-3}.
		\end{align}
		\subsubsection*{Step 1: a bound for \eqref{eq:teo-1}} By Lemma \ref{lem:perfect}, we have that 
		\begin{align}
			\|f - (\alpha \psi + \psi')*u_{\tau,\alpha}\|_{L^\infty[T_1,T_2]}&\leq \Delta_1 +  \theta \cdot 28C  (\alpha+1)\\
			&\leq \Delta_1 +  \theta \cdot 28C   (\alpha_0 + \lu +1).
		\end{align}
		\subsubsection*{Step 2: a bound for \eqref{eq:teo-2}} Using Lemma \ref{lem:leakage}, and the fact that $\|u_{\tau,\alpha_0}\|\leq \theta n$, we see that
		\begin{align}
			\MoveEqLeft \|(\alpha\psi + \psi')* u_{\tau,\alpha} - (\alpha_0\psi + \psi')* u_{\tau,\alpha_0}\|_{\infty}\\ 
			&\leq \|(\alpha\psi + \psi')* (u_{\tau,\alpha} - u_{\tau,\alpha_0})\|_{\infty} + \| (\alpha - \alpha_0)\psi * u_{\tau,\alpha_0}\|_{\infty}\\
			&\leq \|\alpha\psi + \psi'\|_1\|u_{\tau,\alpha} - u_{\tau,\alpha_0}\|_\infty + |\alpha-\alpha_0|\|\psi\|_1 \|u_{\tau,\alpha_0}\|_\infty\\
			&\leq \theta\cdot \lu\cdot 4C \cdot \left( \frac{n}{\alpha_0 - \lu}(\alpha+1) + n\right)\\
			&\leq \theta\cdot \lu\cdot 8C(\alpha_0 + \lu +1) \cdot \frac{n}{\alpha_0 - \lu}.
		\end{align}
		\subsubsection*{Step 3: a bound for \eqref{eq:teo-3}} 
        Consider the Wiener amalgam norm \[\|\alpha_0 \psi + \psi'\|_{W(L^\infty,L^1)}
        := \sum_{k\in\Z} \sup_{x\in [-1/2,1/2]} |(\alpha_0 \psi + \psi')(x-k)|,\]
and note that, by \eqref{eq_A1},
$\|\alpha_0 \psi + \psi'\|_{W(L^\infty,L^1)} \leq K (\alpha_0+1) C$ for a universal constant $K$.
		Applying Lemma \ref{lem:spikes}, we see that for all $t \in \mathbb{R}$,
		\begin{align}
			&|(\alpha_0 \psi + \psi')*(u_{\tau,\alpha_0} - u_{\tau',\alpha_0})(t)|
			 \leq  \int_{\R} |(u_{\tau,\alpha_0} - u_{\tau',\alpha_0})(t-x)|\cdot|(\alpha_0 \psi + \psi')(x)| \,dx\\
			&\qquad= \sum_{k\in\Z} \int_{k-1/2}^{k+1/2} |(u_{\tau,\alpha_0} - u_{\tau',\alpha_0})(t-x)|\cdot |(\alpha_0 \psi + \psi')(x)| \, dx\\
   &\qquad= \sum_{k\in\Z} \sup_{y\in [-1/2,1/2]} |(\alpha_0 \psi + \psi')(y-k)|
   \int_{k-1/2}^{k+1/2} |(u_{\tau,\alpha_0} - u_{\tau',\alpha_0})(t-x)| \, dx\\
			&\qquad\leq \|u_{\tau,\alpha_0} - u_{\tau',\alpha_0}\|_{W(L^1,L^\infty)} \cdot \|\alpha_0 \psi + \psi'\|_{W(L^\infty,L^1)} \\
			&\qquad\leq \su \cdot \theta \cdot \left(\alpha_0 + 4n \right) \cdot K \cdot (\alpha_0+1) \cdot C\\
			& \qquad\leq \su\cdot \theta  \cdot 4K \cdot C \cdot (\alpha_0+n) \cdot (\alpha_0+1).
		\end{align}
		Hence,
		\begin{equation}
			\|(\alpha_0\psi + \psi')*(u_{\tau,\alpha_0} - u_{\tau',\alpha_0})\|_\infty\leq \theta \cdot \su \cdot 4CK(\alpha_0 + 1)(\alpha_0 + n).
		\end{equation}
	Combining this with the results of Steps 1 and 2 we obtain \eqref{eq_xa} (with a new universal constant).
	\end{proof}

	\subsection{Proof of Theorem \ref{thm:main}}
 We now state and prove the following more precise version of Theorem \ref{thm:main}.

 \begin{theo}\label{thm:main_plus}
 Let $f$ have bandwidth $\Omega$ with tolerance $\Delta_\Omega = \theta$. Then the bandwidth-based Ansatz \eqref{eq_ba} satisfies
 \begin{align}\label{eq_xc}
 \|f-f_\Omega\|_{L^\infty[T_1,T_2]} &\leq
 C \theta\Big\{ 1 +  \left(\alpha_0 + \lu +\Omega\right) \Big[1+ n\tfrac{\lu}{\alpha_0 - \lu}
			+\su \cdot \left(\alpha_0 + n\Omega\right)\Big]\Big\}.
\end{align}
In particular, if $\Omega \geq 1$,
 \begin{align}\label{eq_xc_plus}
 \|f-f_\Omega\|_{L^\infty[T_1,T_2]} &\leq
 C \theta \left(\alpha_0 + \lu +\Omega\right)\Big[1+ n\tfrac{\lu}{\alpha_0 - \lu}
			+\su \cdot \left(\alpha_0 + n\Omega\right)\Big].
\end{align}
(Here, $C$ denotes an absolute constant.)
 \end{theo}
 \begin{proof}
 We use a rescaling argument that helps improve the dependence on the various parameters.
 The function $\tilde{f}(x)=\frac{1}{\Omega}f(\frac{x}{\Omega})$ has bandwidth $\Omega=1$ with tolerance $\Delta_1 =   \frac{\theta}{\Omega}$:
		\begin{equation}
			\int_{|\xi|\geq1} |\hat{\tilde{f}}(\xi)| \,d\xi = \int_{|\xi|\geq1} |\hat{f}(\Omega\xi)|\,d\xi = \frac{1}{\Omega} \int_{|u|\geq \Omega} |\hat{f}(u)|\,du \leq \frac{\theta}{\Omega}.
		\end{equation}
		Furthermore, 
		\begin{equation}
			\int_{\Omega t_k}^{\Omega t_{k+1}} \tilde{f}(x) e^{\frac{\alpha}{\Omega} (x-\Omega t_k)}\,dx = \frac{1}{\Omega} \int_{\Omega t_k}^{\Omega t_{k+1}} f\left(\frac{x}{\Omega}\right) e^{\frac{\alpha}{\Omega} (x-\Omega t_k)}\,dx = \int_{t_k}^{t_{k+1}} f(u) e^{\alpha(u-t_k)}\,du = q_k \theta.
		\end{equation} 
        Considering also the behavior of $\tilde{f}$ in the open intervals $(\Omega t_k,\Omega t_{k+1})$ we see that measuring $\tilde{f}$ on the observation window $[0,\Omega T]$ by an integrate-and-fire sampler with firing threshold $\theta$ and leakage parameter $\frac{\alpha}{\Omega}$ produces firing times $\{\Omega t_1,\dots,\Omega t_n\}$ and firing phases $\{q_1,\dots,q_n\}$. 

        We apply Proposition \ref{thm:bw-1} to $\tilde{f}$ and the sampler specified as above. Note that the quantities \eqref{eq_pu} and \eqref{eq_fu} associated with $\tilde{f}$ and $f$ coincide, while $\lu$ and $\su$ are linked by a rescaling factor. Denoting
		$$\tilde{f}_1 = \theta
		\sum_{j=1}^{n}
		\sum_{k=1}^j q_k
		\left(\frac{\alpha_0}{\Omega}\psi +  \psi'\right)
		* \left( e^{\frac{\alpha_0}{\Omega}(\Omega t'_k-\bullet)} 1_{[\Omega t'_j,\Omega t'_{j+1})}\right),$$
		where $t'_{n+1}=\infty$, we have that $\Omega \tilde{f}_1 (\Omega x) = f_\Omega(x)$ and
  \begin{align}
  \|f - f_\Omega\|_{L^\infty[T_1, T_2]} = \|\Omega \tilde{f}(\Omega \cdot) - \Omega \tilde{f}_1(\Omega \cdot)\|_{L^\infty[T_1, T_2]} = \Omega \|\tilde{f} - \tilde{f}_1\|_{L^\infty[\Omega T_1, \Omega T_2]}.
  \end{align} 
  On the other hand, Proposition \ref{thm:bw-1} gives
		\begin{align}
			\|\tilde{f} - \tilde{f}_1\|_{L^\infty[\Omega T_1, \Omega T_2]}
			&\leq C \theta \Big\{ \frac{1}{\Omega} +  \Big(\frac{\alpha_0 + \lu}{\Omega} +1\Big) \times
			\\&\qquad\quad \Big[1+ \lu\cdot  \frac{n}{\alpha_0 - \lu}
			+\su \Omega \cdot \Big(\frac{\alpha_0}{\Omega} + n\Big)\Big]\Big\},
		\end{align}		
  for an absolute constant $C>0$, which readily gives \eqref{eq_xc}.

  Finally, if $\Omega \geq 1$, 
then $\left(\alpha_0 + \lu +\Omega\right) \geq 1$ and the first $1$ in \eqref{eq_xc} can be absorbed by the other terms, by choosing a possibly larger constant $C$. 
\end{proof}

\subsection{Proof of Theorem \ref{coro_1}}\label{sec_coro_1}
We use Theorem \ref{thm:main} with $f=f_j$, $j=1,2$ and $f_\Omega$ as in \eqref{eq_ba} (so, in one case we use exact firing times and in the other one approximate ones). We then obtain bounds valid on the inference windows associated with each signal $f_j$, which contain the inference window associated with the whole class $\mathcal{C}$. As observed after the statement of the theorem, under \eqref{eq_under} the bounds simplify to 
\begin{align}
 \|f_j-f_\Omega\|_{L^\infty[T_1,T_2]} &\leq
 C \theta \Omega \big[1+ n\tfrac{\lu}{\alpha_0 - \lu}+\su n\Omega\big].
\end{align}
Finally, we estimate $\|f_1-f_2\|_{L^\infty[T_1,T_2]} \leq 
 \|f_1-f_\Omega\|_{L^\infty[T_1,T_2]} +  \|f_2-f_\Omega\|_{L^\infty[T_1,T_2]}$ and the conclusion follows.
\qed
 
\section{Examples of model bandwidth}\label{sec_models}

To complement Theorem \ref{coro_1} we provide some examples of approximately bandlimited functions. We present certain families of signal models and discuss which parameters need to be kept uniform in order to have a common bandwidth estimate.

\subsection{Shift-invariant spaces}
The \emph{shift-invariant space} generated by $\{\varphi_1,\dots,\varphi_N\} \subset L^2(\R)$ is
\begin{equation}\label{eq:SIS}
V_\Z(\varphi_1,\dots,\varphi_N):=\overline{\text{span}}\left\{\varphi_i(\cdot- k)\,:\, i=1,\dots,N,\,k\in\Z \right\}.
\end{equation}
One usually assumes that the system $\left\{\varphi_i(\cdot- k)\,:\,i=1,\dots,N,\,k\in\Z^d \right\}$ is a Riesz basis of $V_\Z(\varphi_1,\dots,\varphi_N)$,
		i.e., that there exist two constants $A,B>0$ such that
		\begin{equation}\label{eq_Q1}
			A\|c\|^2_{\ell^2}\leq \Big\|\sum_{i=1}^N \sum_{k\in\Z} c_{i,k} \varphi_i(\cdot-k)\Big\|_{L^2}^2 \leq B \|c\|_{\ell^2}^2,\quad \forall\, c=\{c_{i,k}\}_{i,k} \in \ell^2(\{1,\dots,N\}\times \Z).
		\end{equation}
Often, the generators $\varphi_1, \ldots, \varphi_N$ are not exactly bandlimited, but only approximately so, as is the case for example with the Gaussian function. As a relaxation of bandlimitedness, we assume that the frequency profiles of the generators are enveloped as
\begin{align}\label{eq_Q2}
    |\hat \varphi_i(\xi)|\leq D(1+|\xi|)^{-s-1},\quad\xi\in\R, \qquad i=1,\dots,N,
\end{align}
with $s>0$ and $D>0$. 

Let us consider the collection of all (energy bounded) signals that belong to one such signal model. More precisely, let $\mathcal{C}_{\mathbb{Z}}(A,B,D,s)$ be the collection of all functions $f \in L^2(\mathbb{R})$ with $\|f\|_2 \leq 1$ that belong to $V_\Z(\varphi_1,\dots,\varphi_N)$ for some generators satisfying \eqref{eq_Q1} and \eqref{eq_Q2}.
    
    \begin{prop}\label{propo:example-SIS}
	Every function $f \in \mathcal{C}_{\mathbb{Z}}(A,B,D,s)$ has bandwidth $\Omega$ up to tolerance $\bu=\theta$ provided that
		\begin{equation}\label{eq:omega-sis}
			\lfloor \Omega \rfloor  \geq   
   \Big(\sqrt{\tfrac{N}{A}} \tfrac{2D}{s\theta}\Big)^{\frac{1}{s}}.
		\end{equation}
	\end{prop}
    \begin{rema}
    The bandwidth estimate deteriorates as the lower stability constant $A$ in \eqref{eq_Q1} becomes small, and improves as the degree of frequency concentration $s$ grows.
    \end{rema}
	\begin{proof}[Proof of Proposition \ref{propo:example-SIS}]
		Let $f=\sum_{i=1}^N  \sum_{k\in\Z}c_{i,k} \varphi_i(\cdot -k)$, with $\{c_{i,k}\}_{i,k}\in \ell^2$ and $\|f\|_2 \leq 1$. Then $$\hat f(\xi) = \sum_{i=1}^N \hat{c}_i(\xi) \hat \varphi_i(\xi),\quad \xi\in\R,$$ where $\hat{c}_i(\xi):=\sum_{k\in\Z} c_{i,k} e^{-2\pi i k \xi}$. For $\Omega$ as in \eqref{eq:omega-sis},
		\begin{align}
			\int_{|\xi|>\Omega} |\hat{f}(\xi)| \,d\xi
			&\leq \int_{|\xi|>\Omega}  \sum_{i=1}^N |\hat{c}_i(\xi)| |\hat{\varphi}_i(\xi)| \,d\xi
			\leq D \int_{|\xi|>\Omega}  \sum_{i=1}^N |\hat{c}_i(\xi)| (1+|\xi|)^{-s-1} \,d\xi\\*
   \\*
   &\leq D \int_{\xi>\lfloor \Omega \rfloor}  \sum_{i=1}^N (|\hat{c}_i(\xi)|+|\hat{c}_i(-\xi)|) (1+\xi)^{-s-1} \,d\xi\\
			&\leq D  \sum_{\ell\geq \lfloor \Omega \rfloor}
   \int_{\ell}^{\ell+1}
   (1+\xi)^{-s-1}  \sum_{i=1}^N (|\hat{c}_i(\xi)| + |\hat{c}_i(-\xi)|) \,d\xi\\
    &\leq D  \sum_{\ell\geq \lfloor \Omega \rfloor}(1+\ell)^{-s-1} \int_{0}^1 
    \sum_{i=1}^N (|\hat{c}_i(\xi)|+|\hat{c}_i(-\xi)|)  \,d\xi\\*
			&\leq \frac{2D}{s \lfloor \Omega\rfloor ^{s}} \int_0^1  \sum_{i=1}^N |\hat{c}_i(\xi)|   \,d\xi\\
			&\leq N^{1/2} \frac{2D}{s\lfloor \Omega \rfloor ^{s}} \left( \sum_{i=1}^N \int_0^1  |\hat{c}_i(\xi)|^2 \,d\xi \right)^{1/2}\\
			&\leq N^{1/2} \frac{2 D }{s\lfloor \Omega\rfloor ^{s}} \left( \sum_{i=1}^N \sum_{k\in\Z} |c_{i,k}|^2\right)^{1/2}\\
			&\leq \left(\frac{N}{A}\right)^{1/2} \frac{2D \|f\|_2}{s\lfloor \Omega\rfloor ^{s}}\\
			&\leq \left(\frac{N}{A}\right)^{1/2} \frac{2D}{s\lfloor \Omega\rfloor ^{s}} \leq \theta;
		\end{align}
see also \hyperref[note2]{End Note 2}.
	\end{proof}
\subsection{Spline-like spaces with free nodes}
To further illustrate the notion of approximate bandwidth, we consider a signal of the form
\begin{align}\label{eq_P1}
f(x)=\sum_{k} c_k \varphi(x-\lambda_k)
\end{align}
where the translation nodes $\lambda_k < \lambda_{k+1}$ are arbitrary (and possibly not exactly known). This setup occurs often in applications and is analyzed with Prony-like methods under the assumption that
$|\hat{\varphi}(\xi)|$ is bounded away from $0$ near the origin --- or under even stronger assumptions such as the Strang-Fix conditions. After rescaling, we may assume that
\begin{align}\label{eq_P2}
\tau := \essinf_{\xi \in [-1/2,1/2]} |\hat{\varphi}(\xi)| >0.
\end{align}
As before, the generating function $\varphi$ may not be exactly bandlimited, but satisfies an enveloping condition
\begin{align}\label{eq_P3}
|\hat{\varphi}(\xi)| \leq D(1+|\xi|)^{-s-1}, \qquad \xi \in \mathbb{R}.
\end{align}
We want to argue that all such models share a certain common bandwidth, and are thus subject to the robustness analysis of Theorem \ref{coro_1}. As it turns out, this is possible if the minimal separation between different translation nodes satisfies
\begin{align}\label{eq_P4}
\delta := \inf\{|\lambda_k-\lambda_j|: k\not=j\} > 1.
\end{align}
Under \eqref{eq_P3} and \eqref{eq_P4}, the series in \eqref{eq_P1} converges for $c \in \ell^2$.
Let $\mathcal{C}_{\mathrm{free}}(\tau, \delta, s, D)$
be the class of all functions $f$ as in \eqref{eq_P1} with $\|f\|_2 \leq 1$, $c \in \ell^2$,
\eqref{eq_P2}, \eqref{eq_P3}, and \eqref{eq_P4}.
\begin{prop}\label{prop_2}
Every function in $\mathcal{C}_{\mathrm{free}}(\tau, \delta, s, D)$ has bandwidth $\Omega$ up to tolerance $\bu=\theta$ provided that
\begin{align}\label{eq_P6}
\lfloor\Omega\rfloor \geq C_s \cdot D^{\tfrac{1}{s}} \cdot \big(
{\tau\cdot\theta\cdot \min\{\sqrt{\delta-1},1\} }\big)^{-\tfrac{1}{s}},
\end{align}
where $C_s$ is a constant that depends on $s$.
\end{prop}
\begin{rema}
The proof below provides an explicit constant $C_s$. Important for us is the conclusion that the approximate bandwidth is impacted by the minimal separation allowed between different nodes, and that this effect becomes more moderate as the frequency concentration of $\varphi$ grows (that is, as $s$ becomes large).
\end{rema}
\begin{proof}[Proof of Proposition \ref{prop_2}]
Let $f$ be as in \eqref{eq_P1} with $\|f\|_2 \leq 1$ and $c$ finitely supported (the general case follows by a density argument). Then $\hat{f}(\xi)=m(\xi) \hat{\varphi}(\xi)$ with $m(\xi)=\sum_{k} c_{k} e^{-2\pi i \lambda_k\xi}$. We invoke Ingham's inequality \cite{Ing36, MR2114325} in the quantitative form of \cite[Proposition 2.1 and 2.2]{JamSaba}, which gives
\begin{align}
A \|c\|_{\ell^2}^2 \leq \int_{h}^{h+1} |m(\xi)|^2 \,d\xi \leq B \|c\|_{\ell^2}^2, \qquad h \in \mathbb{R},
\end{align}
with $$A=\begin{cases}
        \frac{\pi^2}{8} \left(\frac{\delta^2-1}{\delta^3}\right),\quad & 1<\delta\leq 2,\\
        \frac{3\pi^2}{64} & \delta\geq 2.
    \end{cases}$$
and $$B=2(\delta+1)/\delta;$$
see \hyperref[note3]{End Note 3} for more details. While $m$ may not be periodic, 
our assumptions allow us to compare its norm on any two intervals of length 1. Indeed, for every $\ell \in \mathbb{Z}$,
\begin{align}
\int_{\ell}^{\ell+1} |m(\xi)|^2\,d\xi &\leq
{\frac{B}{A}} \int_{-1/2}^{1/2} |m(\xi)|^2\,d\xi \leq {\frac{B}{A}} \frac{1}{\tau^2} 
\int_{-1/2}^{1/2} |m(\xi)|^2 |\hat{\varphi}(\xi)|^2\,d\xi
\\
&\leq {\frac{B}{A}} \frac{1}{\tau^2} \|f\|_2^2 \leq{\frac{B}{A}} \frac{1}{\tau^2} \leq
\frac{C^2}{\tau^2\cdot\min\{\delta-1,1\}},
\end{align}
for a constant $C>0$. We now proceed as in the proof of Proposition \ref{propo:example-SIS}:
\begin{align}
			\int_{|\xi|>\Omega} |\hat{f}(\xi)| \,d\xi
			&\leq  D \int_{\xi>\lfloor \Omega \rfloor}  (|m(\xi)|+|m(-\xi)|) (1+\xi)^{-s-1} \,d\xi\\
			&\leq D  \sum_{\ell\geq \lfloor \Omega \rfloor} (1+\ell)^{-s-1}
   \int_{\ell}^{\ell+1} 
   (|m(\xi)| + |m(-\xi)|) \,d\xi\\
    &\leq \frac{2CD}{\tau\cdot\min\{\sqrt{\delta-1},1\}}  
    \sum_{\ell\geq \lfloor \Omega \rfloor}(1+\ell)^{-s-1} 
    \\
    		&\leq \frac{2CD}{\tau\cdot\min\{\sqrt{\delta-1},1\} \cdot s\lfloor \Omega\rfloor ^{s}} \leq \theta,
		\end{align}
if $\Omega$ is as in \eqref{eq_P6} for an adequate constant $C_s$.
\end{proof}

\section{Numerical experiments}\label{sec_num}
We considered a test signal of the form
$f(t) = \sum_{k=1}^{8} c_k \phi(t-t_k)$,
with $\phi$ the B-spline of order $5$ and the nodes $t_k$  chosen uniformly at random within the interval $[0,20]$; see Figure \ref{fig_exp1}. We performed various numerical experiments with $\alpha=0.5$, both for encoding and decoding, i. e., $\Delta_\alpha = 0$. Decoding was performed with the bandwidth-based Ansatz \eqref{eq_ba}, where the frequency cut-off $\psi$ was implemented with a raised-cosine with a roll-off of $0.5$ and bandwidth (reciprocal symbol rate) equal to $1$ \cite{glover2010digital}.
The experiments were carried out with GNU Octave \cite{octave} and can be reproduced with freely available software \cite{cr-code}.

Figure \ref{fig_exp2} shows reconstructed signals for various threshold levels, while the corresponding reconstruction error is plotted in Figure \ref{fig_exp3}. Figure \ref{fig_exp4} illustrates the decay of the reconstruction error as the spike uncertainty decreases. The results are consistent with Theorem \ref{thm:main}. We also performed experiments with positive leakage uncertainty, and observed only a minimal impact on the reconstruction error, even for $\lu \approx \alpha_0$. This stands in contrast to the bound in Theorem \ref{thm:main}, which blows up in that regime. In future work, we expect to further elucidate the apparent remarkable robustness of the bandwidth-based Ansatz under leakage uncertainty. 

	\begin{figure}[ht]
		\centering
		\includegraphics[width=1\textwidth]{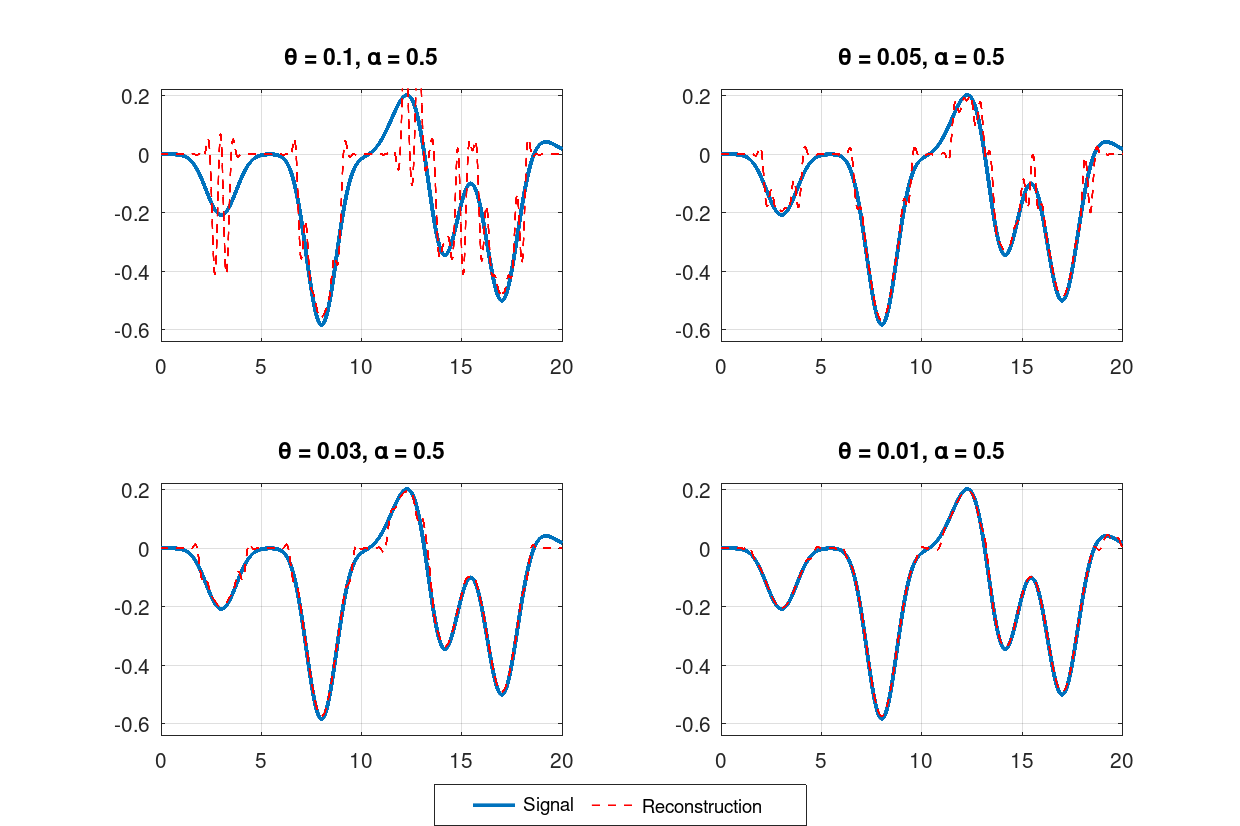}
		\caption{Reconstructions of the test signal with different values of $\theta$.}
		\label{fig_exp2}
	\end{figure}

	\begin{figure}[ht]
	\centering
	\includegraphics[width=0.6\textwidth]{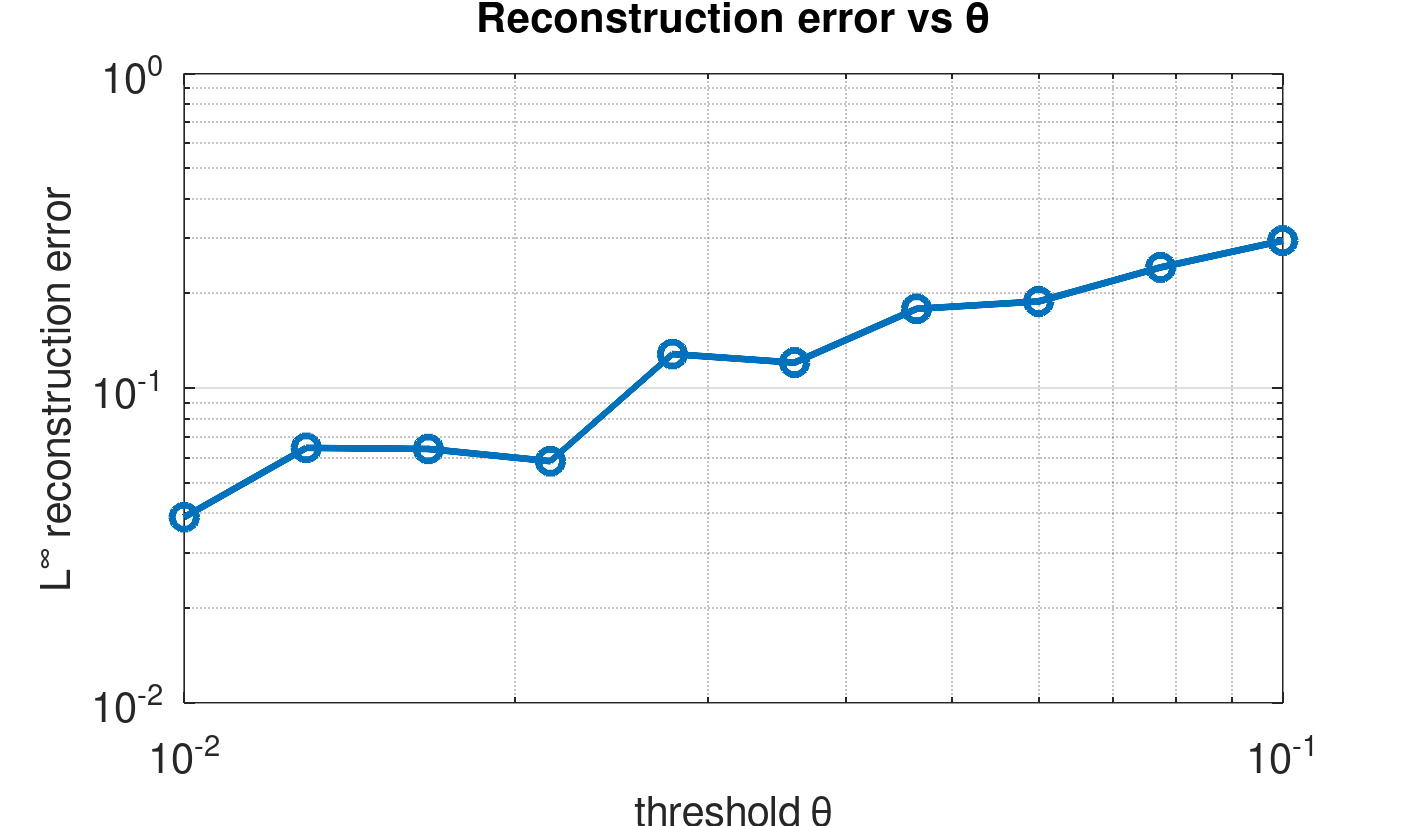}
	\caption{$L^\infty$ reconstruction error for different values of $\theta$.}
	\label{fig_exp3}
\end{figure}

	\begin{figure}[ht]
	\centering
	\includegraphics[width=0.6\textwidth]{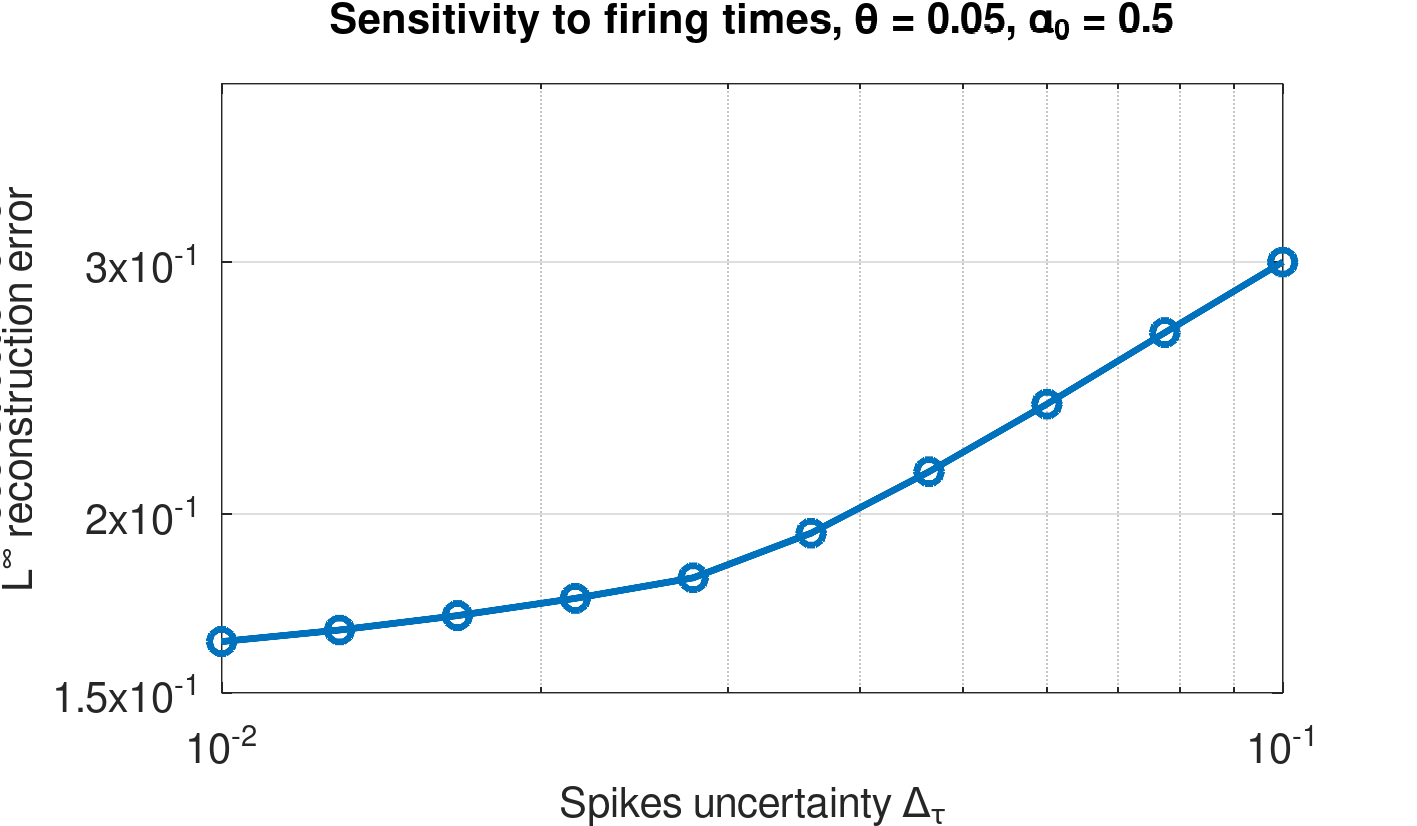}
	\caption{$L^\infty$ reconstruction error for different levels of spike uncertainty.}
	\label{fig_exp4}
\end{figure}

\section{Conclusions}\label{sec_con}
Theorems \ref{thm:main} and \ref{coro_1} show that commonly considered classes of signals can be effectively encoded by an IF sampler as long as the firing threshold $\theta$ is suitably small in comparison to the inverse of the \emph{model approximate bandwidth}, a notion that applies to many commonly used signal models beyond bandlimited functions. Thus, sufficiently different signals cannot have the same or very similar IF outputs. Moreover the conclusions remain true under uncertainties on the leakage parameter and the firing spikes as long as these are comparable in magnitude to the threshold $\theta$.

\section{End notes}\label{sec_notes}
\subsection{End note 1}\label{note1}
A function satisfying conditions \eqref{eq_A1} 
may be obtained as follows: we first let 
$v\in C^{\infty}(\R)$ with $\supp(v)\subset [-1/2,1/2]$, $|\hat{v}(\xi)| \leq C e^{-2 \sqrt{|\xi|}}$ and $\int v = 1$; see, e.g., \cite[Proposition 2.8]{DH98}, and then set $\hat{\psi} := v * 1_{[-3/2,3/2]}$.

\subsection{End note 2}\label{note2} In the proof of Proposition \ref{propo:example-SIS} and \ref{prop_2} we used that for 
$M \in \mathbb{N}_0$:
\begin{align*}
\sum_{\ell\geq M}(1+\ell)^{-s-1}
\leq \sum_{\ell\geq M} \int_{\ell}^{\ell+1} x^{-s-1} \,dx = \int_{M}^{\infty} x^{-s-1} \,dx= \tfrac{1}{s} M^{-s}.
\end{align*}

\subsection{End note 3}\label{note3} In the proof
of Proposition \ref{prop_2}, we applied \cite[Proposition 2.1 and 2.2]{JamSaba} after rescaling and translation. More precisely, since $\lambda_{k+1}-\lambda_{k}\geq \delta > 1$, the cited reference gives
\begin{equation}
    A  \sum_{k\in\Z} |a_k|^2 \leq 
    \frac{1}{\delta}
 \int_{-\delta/2}^{\delta/2}\left| \sum_{k\in\Z} a_k e^{2\pi i \frac{\lambda_k}{\delta} t} \right|^2\,dt =  \int_{-1/2}^{1/2}\left| \sum_{k\in\Z} a_k e^{2\pi i \lambda_k t} \right|^2\,dt \leq B \sum_{k\in\Z} |a_k|^2,
\end{equation}
with $A,B$ as in the proof of Proposition \ref{prop_2}. We can extend the conclusion to every interval $[h,h+1]$ in lieu of $[-1/2,1/2]$; see also \cite{MR1794544}.

\printbibliography
\end{document}